\documentclass[12pt]{amsart}
\usepackage[utf8]{inputenc}
\usepackage{hyperref}

\usepackage{color}
\usepackage {amssymb}
\usepackage{pifont}% http://ctan.org/pkg/pifont
\usepackage {amsmath}
\usepackage {amsthm}

\newcommand {\N}{\mathbb{N}}
\newcommand {\R}{\mathbb{R}}

\newcommand {\be}{\begin{equation}}
\newcommand {\ee}{\end{equation}}

\newtheorem{theorem}{Theorem}[section]
\newtheorem{definition}[theorem]{Definition}
\newtheorem{lemma}[theorem]{Lemma}
\newtheorem{corollary}[theorem]{Corollary}
\newtheorem{proposition}[theorem]{Proposition}
\newtheorem{remark}[theorem]{Remark}

\begin{document}

\title [Weighted Inequalities for the Fractional Laplacian]
{Weighted Inequalities for the Fractional Laplacian and the Existence of Extremals}

\author{Pablo   De N\'apoli}
\address{IMAS (UBA-CONICET) and Departamento de Matem\'atica\\
Facultad de Ciencias Exactas y Naturales\\
Universidad de Buenos Aires\\
Ciudad Universitaria\\
1428 Buenos Aires\\
Argentina}
\email{pdenapo@dm.uba.ar}

\author{Irene Drelichman}
\address{IMAS (UBA-CONICET) \\
Facultad de Ciencias Exactas y Naturales\\
Universidad de Buenos Aires\\
Ciudad Universitaria\\
1428 Buenos Aires\\
Argentina}
\email{irene@drelichman.com}

\author{Ariel Salort}
\address{IMAS (UBA-CONICET) and Departamento de Matem\'atica\\
Facultad de Ciencias Exactas y Naturales\\
Universidad de Buenos Aires\\
Ciudad Universitaria\\
1428 Buenos Aires\\
Argentina}
\email{asalort@dm.uba.ar}

%corregirlo: Poner los subsidios actuales

\thanks{Supported by ANPCyT under grant PICT 2014-1771, by CONICET under grant 11220130100006CO  and by Universidad de Buenos Aires under grants 20020120100050BA and 20020120100029BA. The authors are members of
CONICET, Argentina.}

% corregirlo



\begin{abstract}
In this article we obtain  improved versions of  Stein-Weiss and  Caffarelli-Kohn-Nirenberg inequalities,  involving  Besov norms of negative smoothness. As an application of the former,
we derive the existence of extremals of the Stein-Weiss inequality in certain
cases, some of which are not contained in the celebrated theorem of E. Lieb \cite{Lieb}.
\end{abstract}

\maketitle

\section{Introduction}

In the Euclidean space $\R^n$, it is well-known that negative powers of the
Laplacian admit the integral representation in terms of the Riesz potential
or fractional integral operator: 
$$ 
(-\Delta)^{-s/2} f(x)= I_s(f) = c(n,s) \int_{\R^n} \frac{f(y)}{|x-y|^{n-s}} \; dy
\quad 0 < s < n.
$$

One basic result for this operator is the Stein-Weiss inequality, which gives its behaviour in Lebesgue spaces with 
power weights:

\begin{theorem}\cite[Theorem $B^*$]{SW} \label{SW}
Let $n \geq 1, 0<s<n, 1<p \leq r<\infty, \alpha<\frac{n}{p'},
\gamma>-\frac{n}{r}, \alpha \ge  \gamma $, and
\be \frac{1}{r}=\frac{1}{p}+\frac{\alpha-\gamma-s}{n}. \label{SW-scaling} \ee 
Then,
\be
\| |x|^{\gamma} (-\Delta)^{-s/2} f  \|_{L^r}  \leq C \| |x|^\alpha f  \|_{L^p} \quad \forall f\in L^p(\R^n,|x|^{\alpha p}).\; \label{SW-ineq}
\ee
\end{theorem}

Equivalently, we can rewrite this result as fractional Sobolev
inequality, namely,
\be \| |x|^{\gamma} u  \|_{L^r}  \leq C \| |x|^\alpha (-\Delta)^{s/2} u  
\|_{L^p} \quad \forall u \in \dot{H}^{s,p}_\alpha(\R^n) \label{fractional-Sobolev} \ee
 meaning that we
have a continuous embedding
\be 
\dot{H}^{s,p}_\alpha(\R^n) \subset L^r(\R^n,|x|^{\gamma r})$$
where 
$$ \dot{H}^{s,p}_\alpha(\R^n) = \{ u = (-\Delta)^{-s/2} f : f \in 
L^p(\R^n,|x|^{\alpha p}) \} \label{Sobolev-definition}
\ee
is the weighted
homogeneous Sobolev space of potential
type,
 which is a Banach space with the norm $ \| u \|_{\dot{H}^{s,p}_\alpha} = \| f |x|^\alpha \|_{L^p}.$

\medskip
 
We remark that this embedding is not compact due to the scaling invariance
of the Stein-Weiss inequality. In other words, the scaling condition \eqref{SW-scaling} means that $r$ plays the role of the
critical Sobolev exponent in the weighted setting. 

\medskip 

Our first aim in this work is to obtain an improved version of \eqref{SW-ineq} and \eqref{fractional-Sobolev}. More precisely, for suitable values of the parameters we will prove that
there holds:
\be \| |x|^\gamma (-\Delta)^{-s/2} f  \|_{L^r} \leq 
C \; \| |x|^{\alpha} f  \|_{L^p}^\theta \;  \| f \|_{\dot{B}^{-\mu-s}_{\infty,\infty}}^{1-\theta}
\label{nuestra-desigualdad-1}
\ee
for every $f \in L^p(\R^n,|x|^{\alpha p})  \cap \dot{B}^{-\mu-s}_{\infty,\infty}$, or, equivalently, \be \| |x|^\gamma u \|_{L^r} \leq 
C \; \| |x|^{\alpha} (-\Delta)^{s/2} u  \|_{L^p}^\theta \;  \| u \|_{\dot{B}^{-\mu}_{\infty,\infty}}^{1-\theta}
\label{nuestra-desigualdad-2}
\ee
for every $u\in \dot{H}^{s,p}_\alpha(\R^n) \cap 
\dot{B}^{-\mu}_{\infty,\infty}$, where the Besov norm of negative smoothness is defined in terms of the heat kernel (see Section 2 for a precise definition).

The reader will observe that inequality \eqref{nuestra-desigualdad-2} is reminiscent of the well-known Caffarelli-Kohn-Nirenberg first order interpolation inequality:
\begin{theorem}\cite{CKN}
Assume 
\begin{equation*}
p, q \ge 1, \quad r>0, \quad 0\le \theta \le 1, \quad \frac{1}{p}+\frac{\alpha}{n}, \quad \frac{1}{q}+\frac{\beta}{n}, \quad \frac{1}{r}+\frac{\gamma}{n} >0,
\end{equation*}
where
\begin{equation*}
\label{sigma}
\gamma = \theta\sigma + (1-\theta)\beta.
\end{equation*}
Then, there exists a positive constant $C$ such that the following inequality holds for all $u\in C_0^\infty(\mathbb{R}^n)$
\begin{equation}
\label{ckn}
\| |x|^\gamma u\|_{L^r} \le C \| |x|^\alpha \nabla u \|_{L^p}^\theta \| |x|^\beta u\|_{L^q}^{1-\theta}
\end{equation}
if and only if the following relations hold:
\begin{equation*}
\label{reescale}
\frac{1}{r}+\frac{\gamma}{n} =  \theta \left( \frac{1}{p}+ \frac{\alpha-1}{n}\right) + (1-\theta)\left( \frac{1}{q}+ \frac{\beta}{n}\right)
\end{equation*}
%\bigskip
\begin{align*}
&0 \le \alpha-\sigma& \qquad &\mbox{ if } \, \theta>0, &
%\\ \nonumber &
\\ \nonumber \mbox{and} 
\\ &\alpha-\sigma \le 1&  \qquad &\mbox{ if } \, \theta>0 \quad \mbox{ and } \quad \frac{1}{p}+\frac{\alpha-1}{n}=\frac{1}{r}+\frac{\gamma}{n}.
\end{align*}
\end{theorem}

Indeed, in the local case $s=1$ we will also obtain an improvement of this inequality in some cases, namely, that
\be \| |x|^\gamma u \|_{L^r} \leq 
C \; \| |x|^{\alpha} \nabla u  \|_{L^p}^\theta \;  \| u \|_{\dot{B}^{-\mu}_{\infty,\infty}}^{1-\theta}
\label{nuestra-desigualdad-local}
\ee
holds for  every $u\in  \dot{H}^{1,p}_\alpha(\R^n) \cap 
\dot{B}^{-\mu}_{\infty,\infty} $ for an appropriate range of parameters.

\medskip

Our second aim in this paper is to prove the existence of extremals of inequality \eqref{SW-ineq} by means of a rearrangement-free technique, that allows us to obtain some previously unknown cases.  Let us recall that, by definition, the best constant $S$ in \eqref{SW-ineq} is
\begin{equation} \label{maxim2}
	S=\sup \frac{ \| |x|^\gamma   (-\Delta)^{-s/2} f \|_{L^r}  }{ \| |x|^\alpha f \|_{L^p} }
\end{equation}
where the supremum is taken over all the non vanishing functions $f \in L^p(\R^n, |x|^{\alpha
p})$. 

Best constants and the existence of minimizers/maximizers for
the Stein-Weiss, Sobolev and related inequalities have been studied
extensively in the literature, and it would be impossible to cite all the
references, but we mention some of them that are closely related to our work. 
In a celebrated paper,  E. Lieb  proved  \cite[Theorem 5.1]{Lieb} the existence of minimizers of
the Stein-Weiss inequality (in an equivalent formulation), under the
extra assumptions
$$ p<r,\, \alpha \geq 0, \,\gamma \leq 0. $$
The sign restriction on the exponents in his result comes from the fact that his 
argument is based on a symmetrization (rearrangement) technique.

\medskip

Recently, new rearrangement-free techniques for dealing with these
inequalities in the unweighted case have been introduced in \cite{Frank-Lieb, Frank-Lieb2, PP}. Avoiding
the use of rearrangements could be useful to extend the results to
settings where this technique is not available, for instance when mixed norms are considered (see, for instance, \cite{DL}), or in the setting of stratified Lie groups like in \cite{Ch, Frank-Lieb2}. Indeed,  our results can be extended to the latter setting without essential modifications (replacing $n$ by the
homogeneous dimension of the group, and the Euclidean norm by an homogeneous
norm in the group). However, we have chosen to work in the Euclidean space 
$\R^n$, in order to make our paper accessible to a broader audience. 

\medskip

Improved Sobolev inequalities play an important role for the proofs of existence of maximizers via concentration-compactness arguments. 
We can mention, for instance, \cite[Lemma 4.4]{Frank-Lieb2} where R. L. Frank and E. Lieb obtain an (unweighted) improved inequality with a Besov norm  in the context of the Heisenberg group, which they use derive sharp constants for analogues to the Hardy-Littlewood-Sobolev inequality in that group. 
More recently, the work of G. Palatucci 
and  A. Pisante \cite{PP},  deals with the existence of maximizers in
the unweighted case ($\alpha= \gamma=0$) in $\R^n$, also using and improved Sobolev inequality involving  
a Morrey norm  \cite[Theorem 1.1]{PP}, of which they give two different proofs. One of them, related to our work, is based on the refined
Sobolev inequality of P. Gérard, Y. Meyer and F. Oru \cite{GMO} involving a
Besov norm of negative smoothness, and an embedding result between  Besov
and Morrey spaces \cite[Lemma 3.4]{PP}. 

\medskip

Along these lines, the existence of maximizers of \eqref{maxim2} in the case $\alpha=0$ 
was considered in \cite{Yang}. However, we believe that the  argument in that paper is not correct. 
Indeed, we could not check the validity of inequality (3.2) in 
\cite{Yang}, as the application of the invoked rearrangement inequality would require a decreasing function, that is, a negative 
exponent in the previous inequality. Hence, we don't know whether it is possible to perform the argument
using the refined Sobolev inequality with the Morrey norm in the presence of weights. For this reason, we choose to work directly with the Besov norm and exploit some properties of the heat semigroup. With our improved inequality (Theorem \ref{teorema-besov-infty}), a weighted compactness result  (Proposition \ref{prop.simple}), and the so-called ``method of missing mass'' (invented by E. Lieb in \cite{Lieb}), we can prove the existence of minimizers of the Stein-Weiss inequality only in the case $p=2$ but, in turn, we can have any $\gamma$ in the range $-\frac{n}{r}<\gamma < \alpha$, thus extending the range $\gamma\ge 0$ of \cite{Lieb}.

\medskip

It should be mentioned that there is increasing literature devoted to the study of 
improved versions of the Sobolev-Gagliardo-Nirenberg and related inequalities for their own sake. Besides the above mentioned articles \cite{GMO} and \cite{PP}, we can mention  \cite{Bahouri2005, Ch, Hajaiej2012, Kolyada2014, Ledoux2003, VanSchaftingen2014}, among others. In the proof of our improved \label{nuestra-desigualdad-1} inequality, instead of using the Littlewood-Paley characterization of the Besov space 
(as in \cite{GMO} and \cite{Bahouri2005}), we use a simpler approach inspired by \cite{Ch}, which is based on the thermal definition of the Besov spaces and the representation of the negative powers of the Laplacian in terms of the heat semigroup. 
Besides that, we use the boundedness of the Hardy-Littlewood maximal function with  Muckenhoupt weights,  and the Stein-Weiss inequality. Moreover, our method does not involve 
truncations (as in \cite{Ledoux2003}), a technique which seems not to work in our context due to the non-local character of the fractional Laplacian; and makes no use of rearrangements (as in \cite{Kolyada2014}). 

\medskip

The rest of the paper is organized as follows: in Section 2 we recall the definition of the Besov spaces of negative smoothness and some results on the heat semigroup that will be used in the rest of the paper; in Section 3 we obtain the improved Stein-Weiss inequality \eqref{nuestra-desigualdad-1} and its 
rewritten form \eqref{nuestra-desigualdad-2}, as well as the improved Caffarelli-Kohn-Nirenberg inequality \eqref{nuestra-desigualdad-local}; in Section 4 we prove that the embedding given by \eqref{nuestra-desigualdad-2} is locally compact; and finally in Section 5 we use the method of missing mass and the results of the previous sections to prove the existence of extremals of the Stein-Weiss inequality.

\section{Weighted Estimates for the Heat Semigroup}

In this section we recall the definition of the Besov spaces of negative smoothness and collect some auxiliary results for the heat semigroup,
which will play a central role in our approach.

We recall that the heat semigroup $e^{\Delta t}$ in $\R^n$ is given by
\be e^{t \Delta} f(x_0)= f*h_t(x_0) = \int_{\R^n} f(x) h_t(x_0-x) \; dx \label{heat-semigroup} \ee
where
$$ h_t(x)= \frac{1}{(4\pi t)^{n/2}} \exp \left\{-\frac{|x|^2}{4t} \right\} $$ 
is the heat kernel.

We shall use the following thermic definition of the 
Besov spaces, which goes back to the work of T. Flett \cite{Flett}:

\begin{definition} 
\label{thermal-def}
For any real $\delta>0$ one can define the homogeneous Besov space 
$\dot B_{\infty,\infty}^{-\delta}$ as the space of tempered distributions $u$  on $\R^n$ (possibly modulo polynomials) for which the following norm 
$$
	\|u\|_{\dot B_{\infty,\infty}^{-\delta}} :=\sup_{t>0} t^{\delta/2} \|  e^{\Delta t} f \|_{L^\infty}
$$
is finite.
\end{definition}

 We shall also need the following result:

\begin{proposition}\cite[Proposition 3.2]{Luca} \label{heat-decay}
Let $n\geq 2$, $1 < p \leq q < +\infty$. Assume further that $\alpha, \beta$ satisfy the set of conditions 
 \begin{equation}
    \beta > -\frac nq,\qquad \alpha<\frac{n}{p'}, \qquad
    \alpha \geq \beta.
  \end{equation}
Then the following estimates hold:
\begin{equation}\label{PHeatDer}
     \||x|^{\beta} \partial^{\eta} e^{t\Delta}u \|_{L^{q}} 
     \leq \frac{c_{\eta}}{t^{(|\eta| + \frac{n}{p}-\frac{n}{q} + \alpha-\beta)/2}} 
     \| |x|^{\alpha} u \|_{L^{p}}, \qquad t>0,
     \end{equation}
provided that $|\eta| + \frac{n}{p}-\frac{n}{q} + \alpha-\beta \geq0$,
for each multi-index $\eta$. The range of admissible $p,q$ indices can be relaxed to $1\leq p \leq q \leq +\infty$ provided that $\alpha > \beta $. 
\end{proposition}

Specializing the above result for $q=\infty$ and $\beta=0$ we obtain the following corollaries:

\begin{corollary}
\label{remark-embedding-Besov}
We have an embedding
$$ L^p(\R^n,|x|^{\alpha p}) \subset \dot{B}^{-\mu-s}_{\infty,\infty}$$
with
$$ \mu= \frac{n}{p}+ \alpha-s $$
provided that $\mu>0$ and $0<\alpha<\frac{n}{p'}$. By the lifting property of Besov spaces (see \cite[Section 5.2.3, Theorem 1]{Triebel}), this implies that
$$ 
\dot{H}^{s,p}_\alpha(\R^n) \subset \dot{B}^{-\mu}_{\infty,\infty}.
$$
It follows that, with this choice of $\mu$, 
our inequality \eqref{nuestra-desigualdad-2} is indeed a refinement of the
weighted fractional Sobolev inequality \eqref{fractional-Sobolev}
  and that   
\eqref{nuestra-desigualdad-local} is in some cases a refinement 
of the  Caffarelli-Kohn-Nirenberg inequality \eqref{ckn},  in the sense of \cite{PP}.
\end{corollary}

\begin{corollary} \label{prop2}
Assume that
$$ 0< \alpha < \frac{n}{p^\prime}$$ 
and fix $t>0$. Then
\begin{equation} \label{ecc.1}
 |e^{t\Delta} f(x_0)| \leq C t^{-\frac12 (\frac{n}{p}+\alpha)} \| |x|^\alpha f \|_{L^p} ,
\end{equation}
and
\begin{equation} \label{ecc.2}
 | \partial_{x_j} e^{t\Delta} f(x_0)| \leq C t^{-\frac12 (\frac{n}{p}+\alpha+1)} \| |x|^\alpha f \|_{L^p} \quad j=1,2,\ldots, n 
\end{equation}
for any $x_0$ and any $f \in L^p(\R^n,|x|^{\alpha p})$.
\end{corollary}

\begin{proposition}\label{prop.simple}
	\label{compactness-heat-kernel}
	Let $n\geq 2$, $1 \le p  \le +\infty$. Then, for any fixed $t>0$, the operator $e^{t\Delta}$ is compact from $L^p(\R^n,|x|^{\alpha p})$ to $L^\infty(\R^n)$ provided that
	\begin{equation}
	0<\alpha<\frac{n}{p'}.
	\end{equation}
\end{proposition}

\begin{proof}
	% % % % % % % % % % % % % % %
	Let $\{u_j\}_{j\in\N}$ be a bounded sequence in $L^p(\R^n,|x|^{\alpha p})$, so that
	$$ \| |x|^\alpha u_j \|_{L^p} \leq  C.$$
	Let $v_j= e^{t \Delta} u_j$. Then by Corollary \ref{prop2}, $\{v_j\}_{j\in\N}$ is bounded in $L^\infty(\R^n)$. 
	
	For each $k \in \N$, let us consider the compact set
	$$ C_{k} = \{ x \in \R^n : 0 \leq |x| \leq k \}. $$
	The estimates of Corollary \ref{prop2} also  imply that $\{v_j\}_{j\in\N}$ is equibounded in $C_k$, and so are 
	their first order derivatives, hence $\{v_j\}_{j\in\N}$ is also equicontinuous in $C_k$. Using the Arzel\'a-Ascoli theorem and Cantor's diagonal argument, we conclude that passing 
	again to a subsequence we may assume that
	$$ v_j \to v \; \hbox{uniformly in each} \; C_k. $$
	Since $\alpha>0$ we can choose  $\delta>0$ such that $0<\delta<\alpha$.
	Then, by Proposition \ref{heat-decay} we get
	\begin{align*}
	\sup_{|x|>k} |v_j|  & \leq  \sup_{|x|>k} \left(\frac{|x|}{k}\right)^\delta \;  |v_j|   \\
	&\leq \frac{1}{k^\delta}   \| |x|^\delta v_j  \|_{L^\infty}
	\leq \frac{C}{k^\delta}  \| |x|^\alpha u_j \|_{L^p} \leq \frac{C_1}{k^\delta}
	\end{align*}
	which tends to $0$ as $k \to \infty$, uniformly in $j$. A standard argument gives that $v_j \to v$ strongly in $L^\infty(\R^n)$.
	
\end{proof}

\section{Improved Inequalities}

This section is devoted to  establishing the improved Stein-Weiss and Caffarelli-Kohn-Nirenberg inequalities. 

\begin{theorem}
\label{teorema-besov-infty}
Let
$n\ge 2$, $0<s<n$, $1<p\le r$, $\alpha<\frac{n}{p^\prime}$, $-\gamma<\frac{n}{r}$, 
$\alpha- \frac{\gamma}{\theta} \geq 0$, $\mu>0$,
$ \max\{\frac{p}{r},\frac{\mu}{\mu+s}\} \le \theta \le 1$, and 

\be
\label{reescale-infty}
 \gamma + \frac{n}{r}= \theta \Big(\alpha + \frac{n}{p}-s \Big) + (1-\theta) \mu.
 \ee
Then, for every $f \in L^p(\R^n,|x|^{\alpha p})  \cap \dot{B}^{-\mu-s}_{\infty,\infty}$ 
there holds:
\begin{equation*} 
\| |x|^\gamma (-\Delta)^{-s/2} f  \|_{L^r} \leq 
C \; \| |x|^{\alpha} f  \|_{L^p}^\theta \;  \| f \|_{\dot{B}^{-\mu-s}_{\infty,\infty}}^{1-\theta}
%\label{nuestra-desigualdad-1}
\end{equation*}
or, equivalently, for every $u\in \dot{H}^{s,p}_\alpha(\R^n) \cap 
\dot{B}^{-\mu}_{\infty,\infty} $ 
\begin{equation*} 
 \| |x|^\gamma u \|_{L^r} \leq 
C \; \| |x|^{\alpha} (-\Delta)^{s/2} u  \|_{L^p}^\theta \;  \| u \|_{\dot{B}^{-\mu}_{\infty,\infty}}^{1-\theta}.
%\label{nuestra-desigualdad-2}
\end{equation*}

\end{theorem}

\begin{proof}

Notice that the case $\theta=1$ corresponds to Theorem \ref{SW}, so we may restrict ourselves to the case $\theta<1$.

Let $ u := (-\Delta)^{-s/2} f$, where $f \in L^p(\R^n,|x|^{\alpha p}dx)$. Hence, we write
$$ u =  \frac{1}{\Gamma(s/2)} \int_{0}^\infty t^{s/2-1} e^{t\Delta} f \; dt  $$
and, for fixed $T>0$ to be chosen later, we split the above integral in high and low frequencies, setting
$$ Hf(x) := \frac{1}{\Gamma(s/2)} \int_{0}^T t^{s/2-1} e^{t\Delta} f
\; dt $$
and 
$$ Lf(x) := \frac{1}{\Gamma(s/2)} \int_{T}^\infty t^{s/2-1} e^{t\Delta} f
\; dt.$$

We will obtain pointwise bounds for $Lf$ and $Hf$.

To bound $Lf$, we proceed as in \cite{Ch}, using the thermal definition of Besov spaces (Definition \ref{thermal-def})
 to deduce that 
\be
\label{cota-Lf}
|Lf(x)| \leq C \; \int_T^\infty t^{-\mu/2-1} \; 
 \| f \|_{\dot{B}^{-\mu-s}_{\infty, \infty}} \; dt = 
C T^{-\mu/2} \;  \| f \|_{\dot{B}^{-\mu-s}_{\infty,\infty}} .
\ee

\;

To bound $Hf$, we have to consider different cases, according to whether $\theta=\frac{\mu}{\mu+s}$ or $\theta>\frac{\mu}{\mu+s}$.

\;

{\bf First case: $\theta=\frac{\mu}{\mu+s}.$}

Observe that in this case, we must also have $\theta=\frac{p}{r}$. Indeed, replacing   $\theta=\frac{\mu}{\mu+s}$ in \eqref{reescale-infty} and rearranging terms we have
$$
\frac{\mu+s}{r}-\frac{\mu}{p}=\frac{\alpha \mu}{n} -\frac{\gamma (\mu+s)}{n} \ge 0
$$
where the last inequality follows from the condition $\alpha-\frac{\gamma}{\theta}\ge 0$ and the fact that $\mu >0$. This immediately implies that $\frac{p}{r}\ge \frac{\mu}{\mu+s}$ but, since the reverse inequality holds by hypothesis, we obtain $\frac{p}{r}= \frac{\mu}{\mu+s}=\theta$.

Now, replacing $\theta=\frac{\mu}{\mu+s}=\frac{p}{r}$ in  \eqref{reescale-infty}, we obtain $\alpha p =\gamma r$. This will be useful in what follows.

Having established the relations between the parameters, we remark that this case is contained in \cite[Annexe C]{Ch}, but we outline the result here for the sake of completeness. Following  \cite{Ch}, we obtain
\be
\label{cota-Hf-limite}
|Hf(x)|\le C T^{s/2} Mf(x),
\ee
 where $Mf$ is the Hardy-Littlewood maximal function. 
 
 Now we choose $T$ to optimize the sum of $\eqref{cota-Lf}$ and $\eqref{cota-Hf-limite}$, namely
 $$
 T=\left( \frac{\|f\|_{\dot{B}^{-\mu-s}_{\infty,\infty}}}{Mf} \right)^\frac{2}{\mu+s},
 $$
and we arrive at the pointwise bound 
$$
|(-\Delta)^{-s/2}f| \le C (Mf)^\theta \|f\|_{\dot{B}^{-\mu-s}_{\infty,\infty}}^{1-\theta}.
$$

This implies 
\begin{align*}
\|(-\Delta)^{-s/2}f |x|^\gamma\|_{L^r} & \le C \|(Mf)^\theta |x|^\gamma\|_{L^r} \|f\|_{\dot{B}^{-\mu-s}_{\infty,\infty}}^{1-\theta}\\
& = C \|Mf  |x|^{\gamma r/p}\|_{L^p}^\theta \|f\|_{\dot{B}^{-\mu-s}_{\infty,\infty}}^{1-\theta}\\
& \le C \|f  |x|^\alpha\|_{L^p}^\theta \|f\|_{\dot{B}^{-\mu-s}_{\infty,\infty}}^{1-\theta}
\end{align*}
where we have used the relations between the parameters and the fact that $|x|^{\alpha p}=|x|^{\gamma r }$ belongs to the Muckenhoupt class $A_p$ since $-n<\gamma r = \alpha p<n(p-1)$ by hypothesis, and hence the Hardy-Littlewood maximal function is continuous in $L^p$ with that weight.

\medskip

{\bf Second case: $\theta>\frac{\mu}{\mu+s}.$}

In this case,  observe that
$$ Hf(x)= (K_{s,T} * f)(x) $$
where
$$ K_{s,T}(x)= \frac{1}{\Gamma(s/2)} \int_0^T t^{s/2-1} \; h_t(x) \; dt $$
and
$$ h_t(x)= \frac{1}{(4\pi t)^{n/2}} \exp \left\{-\frac{|x|^2}{4t} \right\} $$
is the heat kernel. 

Now, setting $2\varepsilon =\mu/\theta -\mu >0$ and  noting that $(n-s)/2+\varepsilon>0$, we have that 
$$
	 e^{-x} \leq \frac{C}{x^{(n-s)/2+\varepsilon}} \quad \; \hbox{for} \; 
	x >0
 $$
whence
\begin{align*}
0 \leq K_{s,T}(x) & \leq C \; \int_0^T t^{(s-n)/2-1} 
 \left(\frac{4t}{|x|^2} \right)^{(n-s)/2+\varepsilon} \; dt \\
& \leq C \; \frac{1}{|x|^{n-s+2\varepsilon}} \int_0^T
t^{-1+\varepsilon} \; dt  \\
&\leq C \frac{1}{|x|^{n-s+2\varepsilon}}
T^\varepsilon .
\end{align*}

Hence, 
\be 
\label{cota-Hf-subcritica}
|Hf(x)| \leq C \;  T^\varepsilon \; I_{s-2\varepsilon}f(x) \ee
and to optimize the sum of $\eqref{cota-Lf}$ and $\eqref{cota-Hf-subcritica}$ we have choose 
$$ T= \left( \frac{ \| f
\|_{\dot{B}^{-\mu-s}_{\infty,\infty}}}{I_{s-2\varepsilon}f(x)}
\right)^{1/(\varepsilon+\mu/2)} .$$

Hence, in this case we have the pointwise bound

$$ |(-\Delta)^{-s/2} f(x)|\leq C I_{s-2\varepsilon}f(x)^\theta 
\| f \|_{\dot{B}^{-\mu-s}_{\infty,\infty}}^{1-\theta}. $$

Setting $\tilde{r}=\theta r$, taking $r$-norm and using Theorem \ref{SW}, we have 
\begin{align}
\| |x|^\gamma (-\Delta)^{-s/2} f \|_{L^r} &\leq 
C \; \| |x|^{\gamma/\theta} I_{s-2\varepsilon}f \|_{L^{\tilde{r}}}^\theta \| f
\|_{\dot{B}^{-\mu-s}_{\infty,\infty}}^{1-\theta} \nonumber \\
&\leq 
C \; \| |x|^{\alpha} f \|_{L^p}^\theta \| f \|_{\dot{B}^{-\mu-s}_{\infty,\infty}}^{1-\theta}.
\label{nuestra-desigualdad}
\end{align}
It remains to check the  conditions of Theorem \ref{SW}: $\alpha<\frac{n}{p'}$ and $\alpha-\frac{\gamma}{\theta} \ge 0$  are immediate, while $-\frac{\gamma}{\theta}<\frac{n}{\tilde r},  p\le \tilde r$ and $\frac{1}{\tilde{r}} = \frac{1}{p} +\frac{\alpha-\gamma{/\theta}-(s-2\varepsilon)}{n}$ follow by our choice  of $\tilde r$ and $\varepsilon$ and the hypotheses of our theorem.

This proves our first inequality. To prove  the equivalence with the second one we need to use the lifting property of Besov spaces 
$$ 
\| f
\|_{\dot{B}^{-\mu-s}_{\infty,\infty}} =
\| (-\Delta)^{s/2} u
\|_{\dot{B}^{-\mu-s}_{\infty,\infty}}  = 
\| u
\|_{\dot{B}^{-\mu}_{\infty,\infty}}  
$$
and we arrive at the desired inequality.

\end{proof}

As announced, an analogous result can also be obtained in the local case $s=1$, with the gradient instead of the 
fractional Laplacian. Indeed, Theorem \ref{teorema-besov-infty} implies the following refined 
weighted Sobolev inequality, which is an improvement of the Caffarelli-Kohn-Nirenberg inequalities 
\cite{CKN} in some cases:

\begin{theorem}
\label{teorema-besov-local-case}
Let $n>1$, $1<p\le r$, $\alpha<\frac{n}{p^\prime}$, $-\frac{n}{r}<\gamma<\frac{n}{r^\prime}$, 
$\alpha- \frac{\gamma}{\theta} \geq 0$, $\mu>0$,
$ \max\{\frac{p}{r},\frac{\mu}{\mu+1}\} \le \theta \le 1$, and 

\begin{equation*}
 \gamma + \frac{n}{r}= \theta \Big(\alpha + \frac{n}{p}-1 \Big) + (1-\theta) \mu.
\end{equation*}
Then, for every $u\in  \dot{H}^{1,p}_\alpha(\R^n) \cap 
\dot{B}^{-\mu}_{\infty,\infty} $ 
\begin{equation*}
 \| |x|^\gamma u \|_{L^r} \leq 
C \; \| |x|^{\alpha} \nabla u  \|_{L^p}^\theta \;  \| u \|_{\dot{B}^{-\mu}_{\infty,\infty}}^{1-\theta}.
%\label{nuestra-desigualdad-local}
\end{equation*}
\end{theorem}

\begin{proof}
We consider the classical Riesz transforms $R_j$,
$$ R_j u= \frac{\partial}{\partial x_j} (-\Delta)^{-1/2} u.$$
Since $R_j$ is a Calder\'on-Zygmund operator, it is bounded in $L^r(\R^n,|x|^{\gamma r})$, 
because the weight $|x|^{\gamma r}$ belongs to the Muckenhoupt class $A_r$ by hypothesis. Moreover, using the Fourier 
transform, it is easy to check that
$$ \sum_{j=1}^n R_j^2 = - I.$$
Then, 
$$ \| |x|^\gamma u \|_{L^r} \leq  
\sum_{j=1}^n   \| |x|^\gamma  R_j^2 u \|_{L^r} \leq C \sum_{j=1}^n  
\|  |x|^\gamma  R_j u \|_{L^r}.$$
We apply Theorem \ref{teorema-besov-infty}
 (with $s=1$) in order to obtain 
\begin{align*}
\| |x|^\gamma R_j u  \|_{L^r} & \leq 
C \; \| |x|^{\alpha} (-\Delta)^{1/2} R_j u  \|_{L^p}^\theta \;  \| R_j u \|_{\dot{B}^{-\mu}_{\infty,\infty}}^{1-\theta} \\
& \leq  \; C \; \left \| |x|^{\alpha} \frac{\partial u}{\partial x_j}  \right \|_{L^p}^\theta \;  \| R_j u \|_{\dot{B}^{-\mu}_{\infty,\infty}}^{1-\theta} \\
& \leq  \; C \; \left \| |x|^{\alpha} \frac{\partial u}{\partial x_j}  \right \|_{L^p}^\theta \;  \| u \|_{\dot{B}^{-\mu}_{\infty,\infty}}^{1-\theta}
\end{align*}
since $R_j$ is a bounded operator in the Besov space $\dot{B}^{-\mu}_{\infty,\infty}$ (see \cite{Grafakos}).
Hence, we conclude that 
\begin{align*}
\| |x|^\gamma u \|_{L^r} \leq C \sum_{j=1}^n \left \| |x|^{\alpha} \frac{\partial u}{\partial x_j}  \right \|_{L^p}^\theta \;  \| u \|_{\dot{B}^{-\mu}_{\infty,\infty}}^{1-\theta} \\
\leq C  \left \| |x|^{\alpha} \nabla u  \right \|_{L^p}^\theta \;  \| u \|_{\dot{B}^{-\mu}_{\infty,\infty}}^{1-\theta}
\end{align*}
as announced.
\end{proof}

\section{Local compactness of the embedding}

In this section, we prove a version of the 
Rellich-Kondrachov theorem for the weighted
homogeneous Sobolev space $\dot{H}^{s,p}_\alpha(\R^n)$. We shall need it
in the proof of Theorem \ref{existence-of-maximizers}, but we believe that it could be of
independent interest for the study of other fractional elliptic problems. It is worth 
noting that this result does not seem to follow directly from the standard unweighted version of the compactness theorem, since it is not easy to perfom truncation arguments due to the non-local nature of the fractional Laplacian operator in the definition \eqref{Sobolev-definition} of the space $ \dot{H}^{s,p}_\alpha(\R^n)$. Instead, we work directly with the definition of the weighted space.

\begin{theorem}
	\label{teo.compacidad} 
	Let $n\geq 1$, $0 < s< n$, $1 < p \leq  q <\infty$. Assume further that $\alpha, \beta$ and $s$ satisfy the set of conditions 
	\begin{equation}
	\beta > -\frac {n}{q},\qquad \alpha<\frac{n}{p'}, \qquad
	\alpha \geq \beta
	\end{equation}
	and
	\be \alpha + \frac{n}{p} > s> \frac{n}{p}-\frac{n}{q} + \alpha-\beta >0. \label{subcriticality} \ee
	Then, for any compact set $\mathcal{K} \subset \R^n$, we have the compact embedding
	\be \dot{H}^{s,p}_\alpha(\R^n) \subset L^q(\mathcal{K},|x|^{\beta q}). \label{embedding}  \ee
	\label{thm-local-compactness}
\end{theorem}

We observe that \eqref{subcriticality} is a subcriticality condition. Indeed, under the hypotheses 
of the theorem
$$ s> \frac{n}{p}-\frac{n}{q} $$ 
and if $s<\frac{n}{p}$, this is equivalent to
$$ q < p^*= \frac{np}{n-sp}.$$

\medskip

We first prove the continuity of the embedding \eqref{embedding}. Let $u \in \dot{H}^{s,p}_\alpha(\R^n)$. Then
$$ u =(-\Delta)^{-s/2} f  \;  \hbox{with} \; f \in L^p(\R^n,|x|^{\alpha p}) $$
and
$$ \| u \|_{\dot{H}^{s,p}_\alpha} = \| |x|^\alpha f \|_{L^p(\R^n)}.$$
We define a new exponent $\tilde{q}$ satisfying the Stein-Weiss scaling condition
$$  \frac{n}{\tilde{q}} := \frac{n}{p} + \alpha-\beta -s.$$
From \eqref{subcriticality} it follows that $\tilde{q}>q$. We define:
$$ r:= \frac{\tilde{q}}{q}>1, \tilde{\beta}:=  \frac{\beta q}{\tilde{q}} = \frac{\beta}{r}  $$
so that
$$ \beta > -\frac {n}{q} \Leftrightarrow \tilde{\beta}  > -\frac {n}{\tilde{q}} $$
and apply H\"older's inequality with exponents $r$ and $r^\prime$ to obtain
\begin{align*}
\int_{\mathcal{K}} |u|^q |x|^{\beta q} \; dx &= \int_{\mathcal{K}} |u|^q |x|^{\beta q/r} |x|^{\beta q/r^\prime} \; dx \\
&\leq \left( \int_{\mathcal{K}} |u|^{\tilde{q}} |x|^{\tilde{\beta} \tilde{q}} \; dx \right)^{1/r} 
\left( \int_{\mathcal{K}} |x|^{\beta q} \; dx \right)^{1/r^\prime} \\
& \leq C_{\mathcal{K}} \left( \int_{\R^n} |u|^{\tilde{q}} |x|^{\tilde{\beta} \tilde{q}} \; dx \right)^{1/r}
\end{align*}
since $\beta>-\frac{n}{q}$. Then, using Theorem \ref{SW}
\begin{align*}
\left( \int_{\mathcal{K}} |u|^q |x|^{\beta q} \; dx \right)^{1/q} &\leq C_{\mathcal{K}}^{1/q} \; 
\left( \int_{\R^n} |u|^{\tilde{q}} |x|^{\tilde{\beta} \tilde{q}} \; dx \right)^{1/\tilde{q}} \\
&\leq C \left( \int_{\R^n}  |f|^p  |x|^{\alpha p} \; \right)^{1/p} \\
&= C \| u \|_{\dot{H}^{s,p}_\alpha} .
\end{align*}
This shows that \eqref{embedding} is a continuous embedding.

\medskip

The main difficulty in the proof of the local compactness is that the kernel
$$ K_s^t(x) = C(n,s) |x|^{-(n-s)} $$
of the Riesz potential is not in the dual 
space $L^{p^\prime}(\R^n,|x|^{-\alpha p^\prime})$. For this reason, 
we introduce for $t>0$ the truncated kernels
$$ K_s^t(x) = C(n,s) |x|^{-(n-s)} \chi_{\{|x|>t\}}.
 $$

The following lemma gives a kind of pseudo-Poincaré 
inequality using these kernels.

\begin{lemma}
	Under the conditions of Theorem \ref{teo.compacidad}, set 
	$$\delta = s - \left( \frac{n}{p}-\frac{n}{q} + \alpha-\beta  \right).$$
	Then, for any function, $u \in L^p(\R^n,|x|^{\alpha p})$ and any $t>0$, we have that
	$$ \| (K_{s}^t * u - K_{s}* u) |x|^\beta \|_{L^q} \leq C t^{\delta} \; \| |x|^\alpha u \|_{L^p}.$$
\end{lemma}

\begin{proof}
	Notice that
	$$  K_{s}^t * f - K_{s}* f(x)= C(n,s) \int_{|x-y|\leq t} \frac{f(y)}{|x-y|^{n-s}} \; dy  .$$
	Hence, since $\delta>0$ by \eqref{subcriticality},
	\begin{align*}
	|K_{s}^t * f - K_{s}* f(x)| &\leq C(n,s) \int_{|x-y| \leq t} \frac{|f(y)|}{|x-y|^{n-s}} \; dy  \\
	&\leq C(n,s) \int_{|x-y|\leq t} \frac{|f(y)|}{|x-y|^{n-s}} 
	\left( \frac{t}{|x-y|} \right)^\delta dy \\
	&\leq C(n,s) \; t^\delta \; \int_{\R^n} \frac{|f(y)|}{|x-y|^{n-(s-\delta)}} \; dy 
	\end{align*}
	and the lemma follows from the Stein-Weiss inequality since $s-\delta>0$ (the definition of $\delta$ means 
	that the required scaling-condition holds with $s-\delta$ in place of $s$).
	%$$ 0 = s - \delta -  \frac{n}{p} +\frac{n}{q} -  \alpha + \beta   $$
	%$$ \frac{n}{q} = \frac{n}{p} + \alpha - \beta- (s-\delta)  $$
\end{proof}

\bigskip

Now we are ready to prove the compactness of the embedding \eqref{embedding}:
let $(u_k)$ be a bounded sequence $\dot H^{s,p}_\alpha(\R^n)$. We may write it as
$$ u_k = (-\Delta)^{-s/2} f_k = K_s * f_k $$
where $f_k$ is a bounded sequence in $L^p(\R^n,|x|^{\alpha p})$. By
reflexivity, passing to a subsequence, we may assume that 
$$ f_k \rightharpoonup f \; \hbox{weakly in} \; L^{p}(\R^n,|x|^{\alpha p}).$$ 
We consider the functions
$$ u_k^t = K_s^t *  f_k $$
$$ u^t = K_s^t *  f $$
and we write
$$ \| (u_k - u) |x|^\beta  \|_{L^q(\mathcal{K})} \leq
\| (u_k - u_k^t) |x|^\beta  \|_{L^q(\mathcal{K})} 
+ \| (u_k^t - u^t) |x|^\beta  \|_{L^q(\mathcal{K})}
+ \| (u^t -u) |x|^\beta  \|_{L^q(\mathcal{K})} . $$

We observe that
$$ \| (u_k - u_k^t) |x|^\beta  \|_{L^q(\mathcal{K})} \leq t^{\delta} \| f_k |x|^\alpha \|_{L^p}
\leq C t^\delta $$
and that
$$ \| (u^t - u) |x|^\beta  \|_{L^q(\mathcal{K})} \leq t^{\delta} \| f |x|^\alpha \|_{L^p} 
\leq C t^\delta.$$

Hence, given $\varepsilon>0$ we can make this two terms less than $\frac{\varepsilon}{3}$ 
for all $k$, provided that we fix $t$ small enough. 

We check that $K_s^t(x_0-\cdot )$ is in $L^{p^\prime}(\R^n,|x|^{-\alpha p^\prime})$. For that,
we consider the integral.
\begin{align*}
I(x) :=& \int_{\R^n} |K_s^t(x-y)|^{p^\prime} |y|^{-\alpha p^\prime} \; dy \\
=& \; C(n,s)^{p^\prime} \int_{|x-y|>t} \frac{1}{|x-y|^{(n-s)p^\prime}}
|y|^{-\alpha p^\prime} \; dy.
\end{align*}

The integrability  condition at zero is
$$ -\alpha p^\prime + n >0 \Leftrightarrow \alpha < \frac{n}{p^\prime} \;  $$
and at infinity is
$$ -\alpha p^\prime -(n-s)p^\prime +n <0 \Leftrightarrow s < \alpha + \frac{n}{p}. $$
Hence $I(x)$ is finite. We conclude that, that for any fixed $t>0$, and any fixed $x_0$ we have that
$$ u_{k}^t(x_0) = K_s^t * f_k(x_0) \to K_s^t * f(x_0) = u^t(x_0).$$

Moreover, we have that for any $x_0$ in the compact set $\mathcal{K}$
\be \int_{\R^n} |K_s^t(x_0-x)|^{p^\prime} |x|^{-\alpha p^\prime} \; dx \leq C_{\mathcal{K}}. \label{local-bound}\ee
Indeed, if we write
$$ I(x)= I_0(x) + I_\infty(x) $$
where $\mathcal{K} \subset B(0,R)$ and 
$$ I_0(x) = C(n,s)^{p^\prime} \int_{|x-y|>t,|y|\leq 2R} \frac{1}{|x-y|^{(n-s)p^\prime}}
|y|^{-\alpha p^\prime} \; dy,$$

$$ I_\infty(x) = C(n,s)^{p^\prime} \int_{|x-y|>t,|y|> 2R} \frac{1}{|x-y|^{(n-s)p^\prime}}
|y|^{-\alpha p^\prime} \; dy,$$
then,
$$ I_0(x) \leq C(n,s)^{p^\prime} \int_{y|\leq 2R} \frac{1}{t^{(n-s)p^\prime}}
|y|^{-\alpha p^\prime} \; dy = C(n,s)^{p^\prime} \frac{1}{t^{(n-s)p^\prime}} (2R)^{-\alpha p^\prime+n}.$$
On the other hand, when $x \in \mathcal{K}$ and $y$ is in the integration region of $I_\infty$
$$ |x|\leq R < \frac{1}{2} |y| $$
and
$$ |x-y| \geq |y|-|x| \geq |y| - \frac{1}{2} |y| = \frac{1}{2} |y|.$$
Hence,
\begin{align*} I_\infty(x) &\leq  C(n,s)^{p^\prime} \int_{|y|> 2R} \frac{1}{\left(\frac{1}{2}|y|\right)^{(n-s)p)^\prime}}
|y|^{-\alpha p^\prime} \; dy \\
&= \tilde{C}(s,n,p) \; R^{-(n-s)p^\prime-\alpha p^\prime + n} 
\end{align*}
which implies \eqref{local-bound}. Thus, $(u^t_k)$ is uniformly
bounded on $\mathcal{K}$, and by the bounded convergence theorem,
$$ \| (u_k^t - u^t) |x|^\beta  \|_{L^q(\mathcal{C})} \to 0 $$
as $k \to \infty$ (since the condition $\beta > -\frac nq$ means that the weight 
$|x|^{\beta q}$ is integrable on $\mathcal{K})$. Therefore, we can make it less than $\varepsilon/3$ for $k \geq k_0(\varepsilon)$.

We conclude that $u_k \to u$ strongly in $L^q(\R^n,|x|^{\beta q})$ as we wanted.

\section{Existence of maximizers of the Stein-Weiss Inequality}

In this section we prove our main theorem, which extends the result of Lieb  \cite[Theorem 5.1]{Lieb} to some previously unknown cases when $p=2$. 

The proof uses a well-known strategy, but the results are new thanks to our improved Stein-Weiss inequality (Theorem \ref{teorema-besov-infty}),  and the weighted compactness results  (Proposition \ref{prop.simple} and Theorem \ref{teo.compacidad}). First, we show that from any maximizing sequence we can extract -after a suitable rescaling- a subsequence with a non-zero weak limit. In the second part, we use the so-called ``method of missing mass'' (invented by E. Lieb in \cite{Lieb}) to prove that such a limit is actually an optimizer.   

\begin{theorem} \label{existence-of-maximizers}
Assume that $n\geq 2$, $0<s<\frac{n}{2}$, $2<r<\infty$, $0<\alpha<\frac{n}{2}$, $-\frac{n}{r}<\gamma < \alpha$ and that the relation
$$
	\frac{1}{r}-\frac{1}{2}=\frac{\alpha-\gamma-s}{n}
$$
holds. Then, there exists a maximizer for $S$.
\end{theorem}

\begin{remark}
Notice that condition $0<s<\frac{n}{2}$ does not appear explicitly in  \cite[Theorem 5.1]{Lieb} but is implied by the other conditions on the parameters. Indeed, since $\alpha\ge \gamma$, we have that $1<r\le \frac{2n}{n-2s}$ and, in particular, there must hold $n-2s> 0$.
\end{remark}

\begin{proof}[Proof of Theorem \ref{existence-of-maximizers}]
	Let $\{f_k\}_{k\in\N}$ be a maximizing sequence of $S$, that is, 
	\begin{equation} \label{minim}
		\| |x|^\alpha  f_k \|_{L^2} = 1  \quad\mbox{and}\quad \| |x|^\gamma  (-\Delta)^{-s/2} f_k \|_{L^r} \to S .
	\end{equation}
	By Corollary \ref{remark-embedding-Besov}, if we set
	\begin{equation} \label{mu.1}
		\mu  = \frac{n}{2} + \alpha -s,
	\end{equation}
	it holds that
	$$ 
		\| f_k \|_{\dot{B}^{-\mu-s}_{\infty,\infty}} \leq
		C \| |x|^\alpha f_k \|_2 = C.
	$$
	On the other hand, Theorem \ref{teorema-besov-infty}
 gives that
	$$
		  \| f_k \|_{\dot{B}^{-\mu-s}_{\infty,\infty}} \geq C >0 
	$$
	provided we choose $\theta$ such that 
	$$
	\max \left\{ \frac{p}{r}, \frac{\mu}{\mu+s}, \frac{\gamma}{\alpha} \right\} < \theta <1
	$$
	which is possible by the hypotheses of the Theorem.
	This means that
	$$ 
		\sup_{t>0} \;  t^\frac{\mu+s}{2} \| e^{t\Delta} f_k \|_{L^\infty} \geq C>0.
	$$ 
	Consequently, for each $k \in \N$ we can find $t_k>0$ such that
	$$ 
		t_k^\frac{\mu+s}{2} \; \|  e^{t_k\Delta} f_k \|_{L^\infty} \geq \frac{C}{2} > 0.
	$$	
	Now we set
	\begin{equation} \label{reescale}
	 \tilde{f}_k(x)= t_k^{\frac{1}{2}(\frac{n}{2}+\alpha)} \;  f_k(t_k^\frac12 x),
	\end{equation}	
	and observe that, by parabolic scaling, 
	$$
	 e^{1.\Delta} \tilde f_k (x) = t^{\frac12 (\frac{n}{2}+\alpha)} e^{\Delta t_k} f_k( t_k^{\frac12} x).
	$$	
	Then,
	\begin{align} \label{ectt}
	\begin{split}
		\|  e^{1.\Delta} \tilde f_k\|_{L^\infty}&=t_k^{\frac{1}{2}(\frac{n}{2}+\alpha)}
		\| e^{t_k\Delta}   f_k(t_k^\frac12 x)\|_{L^\infty}\\
		&=t_k^\frac{\mu+s}{2}  \|   e^{t_k \Delta}   f_k\|_{L^\infty} \geq \frac{C}{2}>0
	\end{split}		
	\end{align} 	
	since relation \eqref{mu.1} holds.

	Observe that, in view of the scaling invariance of the $L^2(\R^n,|x|^{2 \alpha})$ norm, the sequence $\{\tilde f_k\}_k$ is bounded in $L^2(\R^n,|x|^{2\alpha })$ and,  that the maximization problem \eqref{maxim2} is invariant under the rescaling given by $\tilde f_k$ as long as
	$$
		\gamma=\frac{n}{2}-\frac{n}{r}+\alpha-s,
	$$
	which holds by our assumptions. Indeed, 
	\begin{equation} \label{acotado}
		 \| |x|^\alpha \tilde{f}_k\|_{L^2} = \| |x|^\alpha f_k\|_{L^2}=1.
	\end{equation}
	Consequently, $\{\tilde f_k\}_k$ is also a minimizing sequence of $S$, that is, 
	\begin{equation}  \label{minimftilde}
		\| |x|^\alpha  \tilde f_k \|_{L^2} = 1  \quad\mbox{and}\quad \| |x|^\gamma  (-\Delta)^{-s/2} \tilde f_k \|_{L^r} \to S .
	\end{equation}

	It remains to show that there exists $g\not\equiv 0$ such that $\tilde f_k\to g$ strongly in $L^2(\R^n,|x|^{2\alpha  })$. The last requirement will allow us to deduce that $g$ is also a minimizer for $S$, i.e., 
	\begin{equation}  \label{minimg}
		\| |x|^\alpha    g \|_{L^2} = 1  \quad\mbox{and}\quad \| |x|^\gamma  (-\Delta)^{-s/2}   g \|_{L^r} = S .
	\end{equation}	
	
	By reflexivity, from \eqref{acotado} there exists $g\in L^2(\R^n,|x|^{2\alpha  })$ and a subsequence still denoted by $\tilde f_k$ such that 
	\begin{equation} \label{conv.debil}
		\tilde{f}_{k} \rightharpoonup g \quad \mbox{ weakly in } L^2(\R^n, |x|^{ 2\alpha  }).
	\end{equation}
	We set
	$$
		u_k := (-\Delta)^{-s/2} \tilde f_k, \quad w := (-\Delta)^{-s/2} g.
	$$
	From Proposition \ref{prop.simple},   the compactness of the operator $e^{1.\Delta}$ implies that
	$$ e^{1.\Delta} \tilde{f}_{k} \to e^{1.\Delta} g \;\quad  \hbox{strongly in} \; L^\infty(\R^n) $$
	and then $g\not\equiv 0$, since by \eqref{ectt} we have
	$$
		 \| e^{1.\Delta} g \|_{L^\infty} \geq \frac{C}{2} > 0.
	$$
	
	By Theorem \ref{teo.compacidad} with $2<q<r$ and $\alpha=\beta$, for any compact set $K$ we have 	the compact embedding 
	$$ 
		\dot{H}^{s,2}(\R^n) \subset L^{q}(K) 
	$$
	which implies that passing to a subsequence we may assume that
	$$ 
		u_{k} \to w \quad \mbox{strongly in } L^q(K)
	$$
	and, therefore, up to a subsequence,  $u_{k} \to w $ a.e. in $K$. By using a diagonal argument we obtain that, again, up to a subsequence, 
	\begin{equation} \label{conve}
		u_k \to w \quad \mbox{ a.e. } \R^n.
	\end{equation}
	
	Let us prove that \eqref{minimg} holds. 	Since $u_k\to w$ a.e. $\R^n$, the Brezis-Lieb Lemma (\cite[Lemma 2.6]{Lieb} and \cite{Brezis-Lieb}) claims that
	$$
		\lim_{k\to\infty}\left( \int_{\R^n} |u_k|^r |x|^{r\gamma}\,dx - \int_{\R^n} |u_k -w |^r |x|^{r\gamma}\,dx \right) = \int_{\R^n} |w|^r |x|^{r\gamma}\,dx,
	$$
	but, from \eqref{minimftilde},
	\begin{equation} \label{BL}
		\lim_{k\to\infty} \int_{\R^n} |u_k|^r |x|^{r\gamma}\,dx  =   \int_{\R^n} |w|^r |x|^{r\gamma}\,dx  + \lim_{k\to\infty}  \int_{\R^n} |u_k -w |^r |x|^{r\gamma}\,dx.	
	\end{equation}

	Combining \eqref{BL} with \eqref{minimftilde} and the elementary inequality 
	\begin{equation} \label{des}
		a^\frac{r}{2}+b^\frac{r}{2} \leq (a+b)^\frac{r}{2}
	\end{equation}
	 for $a,b\geq 0$ and $r>2$, we have
	\begin{align*}
		S^r&=\lim_{k\to\infty} \int_{\R^n} |u_k|^r |x|^{r\gamma}\,dx  =   \int_{\R^n} |w|^r |x|^{r\gamma}\,dx  + \lim_{k\to\infty}  \int_{\R^n} |u_k -w |^r |x|^{r\gamma}\,dx\\
		&\leq S^r\left(  \int_{\R^n} |g|^2 |x|^{2\alpha}\,dx  \right)^\frac{r}{2} + S^r \left( \limsup_{k\to\infty}  \int_{\R^n} |\tilde f_k -g|^2 |x|^{ 2 \alpha } \, dx\right)^\frac{r}{2}\\
		&\leq S^r \left (  \int_{\R^n} |g|^2 |x|^{2 \alpha }\,dx  + \limsup_{k\to\infty}  \int_{\R^n} |\tilde f_k -g|^2 |x|^{2\alpha} \, dx  \right)^\frac{r}{2} =S^r,
	\end{align*}
where we have used that
	\begin{align} \label{ec1}
	\begin{split}
	\int_{\R^n} |g|^2 |x|^{2 \alpha} &\,dx + 
\limsup_{k\to\infty}  \int_{\R^n} |\tilde f_k - g|^2 |x|^{2\alpha } \,dx\\ 
&= \limsup_{k\to\infty} \int_{\R^n} |\tilde f_k|^2 |x|^{2\alpha}  \,dx=1
	\end{split}
	\end{align}
	since $\tilde{f}_{k} \rightharpoonup g$ weakly in $L^2(\R^n, |x|^{ 2\alpha})$.

	Observe that \eqref{des} is a strict inequality unless $a=0$ or $b=0$. Hence, since all the previous inequalities are in fact equalities, we obtain that $\||x|^\alpha g\|_{L^2}=1$ and $\tilde f \to g$ strongly in $L^2(\R^n,|x|^{2\alpha})$.
	
	By Theorem \ref{SW}, $(-\Delta)^{-s/2}$ is a continuous operator from $L^2(\R^n,|x|^{2\alpha})$ into $L^r(\R^n,|x|^{\gamma r})$ and then
	$$
		u_k \to w \quad \mbox{ strongly in } L^r(\R^n, |x|^{\gamma r})
	$$
	from where \eqref{minimg} follows. The proof is now complete.

\end{proof}

\bibliographystyle{amsplain}
\bibliography{biblio}
 
\end{document}